\documentclass{amsart}
\usepackage{tikz-cd}
\usepackage{xcolor,colortbl}
\usepackage{latexsym,amssymb}
\usepackage[english]{babel}
\usepackage{bm}
\usepackage{amsfonts, amsthm}
\usepackage{amsmath}
\usepackage{mathrsfs}
\usepackage{multirow}
\usepackage{bm}
\usepackage{pgf,tikz}
\usepackage{hyperref}

\newtheorem{prop}{Proposition}[section]
\newtheorem{lem}[prop]{Lemma}
\newtheorem{cor}[prop]{Corollary}
\newtheorem{thm}[prop]{Theorem}

\theoremstyle{definition}
\newtheorem{rem}[prop]{Remark}
\newtheorem{defi}[prop]{Definition}
\newtheorem{ex}[prop]{Example}
\newtheorem*{KN}{Crazy Knight's Tour Problem}

\def\Z{\mathbb{Z}}

\numberwithin{equation}{section}

\def\Z{\mathbb{Z}}

\newcommand{\probname}{Crazy Knight's Tour Problem}

\def\H{\mathrm{H}}

\def\W{\mathrm{W}}

\def\E{\mathcal{E}}

\def\R{\mathcal{R}}
\def\C{\mathcal{C}}

\usepackage{mathtools}

\DeclarePairedDelimiter\floor{\lfloor}{\rfloor}

\begin{document}

\title[Weak Heffter arrays]{Weak Heffter Arrays and biembedding graphs on non-orientable surfaces}

\author[S. Costa]{Simone Costa}
\address{DICATAM - Sez. Matematica, Universit\`a degli Studi di Brescia, Via
Branze 43, I-25123 Brescia, Italy}
\email{simone.costa@unibs.it}

\author[L. Mella]{Lorenzo Mella}
\address{Dip. di Scienze Fisiche, Informatiche, Matematiche, Universit\`a degli Studi di Modena e Reggio Emilia, Via Campi 213/A, I-41125 Modena, Italy}
\email{lorenzo.mella@unipr.it}

\author[A. Pasotti]{Anita Pasotti}
\address{DICATAM - Sez. Matematica, Universit\`a degli Studi di Brescia, Via
Branze 43, I-25123 Brescia, Italy}
\email{anita.pasotti@unibs.it}


\maketitle

\begin{abstract}
In 2015, Archdeacon proposed the notion of Heffter arrays in view of its connection to several other combinatorial objects.
In the same paper he also presented the following variant.
A \emph{weak Heffter array} $\W\H(m,n;h,k)$ is an $m \times n$ matrix $A$  such that:
 each row contains $h$ filled cells and each column contains $k$ filled cells;
 for every $x \in \mathbb{Z}_{2nk+1} \setminus \{0\}$, there is exactly one cell of $A$ whose element is one of the following:
$x,-x,\pm x,\mp x$, where the upper sign on $\pm$ or $\mp$ is the row sign  and
the lower sign is the column sign;
 the elements in every row and column (with the corresponding sign) sum  to $0$ in $\mathbb{Z}_{2nk+1}$.
Also the ``weak concept'', as the classical one, is related to several other topics, such as difference families, cycle systems
and biembeddings.

Many papers on Heffter arrays have been published, while no one on weak Heffter arrays has been written. This is the first one and here we explore necessary conditions, existence and non-existence results, and connections to biembeddings into non-orientable surfaces.
\end{abstract}

\section{Introduction}
In 2015, Archdeacon \cite{A} introduced a class of partially filled (p.f. for short) arrays as an interesting link between combinatorial designs and topological graph theory.  Since then, there has been a good deal of interest in Heffter arrays as well as in related topics such as the sequencing of subsets of a group, biembeddings of cycle systems on a surface, and orthogonal cycle systems.

\begin{defi}\label{def:Heffter}\cite{A}
 A \emph{Heffter array} $\H(m,n;h,k)$ is an $m \times n$ matrix with entries from $\mathbb{Z}_{2nk+1}$ such that:
\begin{itemize}
\item[{\rm (a)}] each row contains $h$ filled cells and each column contains $k$ filled cells;
\item[{\rm (b)}]  for every $x \in \mathbb{Z}_{2nk+1} \setminus \{0\}$, either $x$ or $-x$ appears in the array;
\item[{\rm (c)}] the elements in every row and column sum to $0$ in $\mathbb{Z}_{2nk+1}$.
\end{itemize}
\end{defi}
To date several variants and generalizations of this concept have been introduced (see \cite{DP}), so
we will refer to these arrays as \emph{classical}.

In \cite{A}, Archdeacon also proposed the following variant of the previous concept.

\begin{defi}\label{def:weak}
 A \emph{weak Heffter array} $\W\H(m,n;h,k)$ is an $m \times n$ matrix $A$  such that:
\begin{itemize}
  \item[$(\rm{a_1})$] each row contains $h$ filled cells and each column contains $k$ filled cells;
\item[$(\rm{b_1})$]  for every $x \in \mathbb{Z}_{2nk+1} \setminus \{0\}$, there is exactly one cell of $A$ whose element is one of the following:
$x,-x,\pm x,\mp x$, where the upper sign on $\pm$ or $\mp$ is the row sign  and
the lower sign is the column sign;
\item[$(\rm{c_1})$] the elements in every row and column (with the corresponding sign) sum  to $0$ in $\mathbb{Z}_{2nk+1}$.
\end{itemize}
\end{defi}

If $m=n$, then $h=k$ and we use the notation $\H(n;k)$ instead of $\H(n,n;k,k)$.
Furthermore, a rectangular array with no empty cells $\H(m,n;n,m)$ is denoted by $\H(m,n)$.
Analogous notation holds for a weak Heffter array.

A (weak) Heffter array is said to be \emph{integer} if condition {$\rm (c)$} of Definition \ref{def:Heffter}
(condition {$\rm (c_1)$} of Definition \ref{def:weak})
 is strengthened so that the elements in every row and every column, seen as integers in $\pm \{1,\ldots,nk\}$, sum to zero in $\Z$.

\begin{ex}\label{ex:WeakA}
The following is a weak Heffter array $\W\H(3,4)$ over $\mathbb Z_{25}$ (see Figure 6 of \cite{A}):
$$
\begin{array}{|c|c|c|c|} \hline
1 & -7 & -6 & 12 \\ \hline
2 & -4 & \pm 10 & \mp 8 \\ \hline
-3 & \mp 11 & \pm 9 & 5 \\ \hline
\end{array}
$$
\end{ex}

The existence of classical Heffter arrays has been largely investigated. The most important results are the following.
\begin{thm}\label{thm:Heffter}{\rm \cite{ADDY,CDDY,DW}}
There exists an $\H(n; k)$ if and only if $n \geq  k \geq  3$.
\end{thm}

\begin{thm}{\rm \cite{ABD}}\label{th:existence}
There exists an $\H(m,n)$ if and only if $m,n\geq 3$.
\end{thm}
Recent partial results for integer rectangular Heffter arrays with empty cells can be found in \cite{MP3}.

On the other hand, no result has been obtained about weak Heffter arrays, actually, this
is the first paper studying this class of arrays.
As shown into details in Section \ref{sec:2}, one can get a weak Heffter array starting from a classical one
simply by placing different row and column signs in a suitable set of cells, see for instance, Example
\ref{ex:WeakFromClass}. Clearly, not all weak Heffter arrays can be constructed starting from a classical one.
Indeed the one proposed by Archdeacon, see Example \ref{ex:WeakA},
that as far as we know is the only weak Heffter array considered, cannot be obtained from a classical $\H(3,4)$
by placing different row and column signs in a suitable set of cells.
This induces us to introduce the following definition.
\begin{defi}
A weak $\W\H(m,n;h,k)$ is said to be \emph{strictly weak} if it
cannot be obtained from a classical $\H(m,n;h,k)$
by placing different row and column signs in a suitable set of cells.
\end{defi}
In other words, a weak Heffter array $A=(a_{ij})$ is strictly weak if there is no Heffter array $B=(b_{ij})$
with the same parameters such that $|a_{ij}|=|b_{ij}|$, for any $i,j$. Hence it is easy to see that the
$\W\H(3,4)$ of Example \ref{ex:WeakA} is strictly weak.
The terminology \emph{strictly weak} may induce the reader to believe it is easier to construct
these arrays than the classical ones. However, as shown in Section  \ref{sec:3}, this is not the case.
Notably, for some values of $n$ and $k$, an $\H(n;k)$ exists while a strictly weak $\W\H(n;k)$ does not exist.
The motivation for which Archdeacon introduced this variant and for which we believe it is worth studying
is that, as shown in Section 6.1 of \cite{DP}, also these arrays give rise to orthogonal cyclic cycle systems
and, as hinted in \cite{A} and explained in details in Section  \ref{sec:5}, starting from these arrays the existence of embeddings follows.

In Section \ref{sec:2} we investigate the connection between classical and (not strictly) weak Heffter arrays,
showing how some existence results on the classical case apply to this context.
In Sections \ref{sec:3} and \ref{sec:4} we consider and adjust to the weak context a more general concept, introduced in \cite{CMPP}, which admits classical Heffter arrays as a particular case: the \emph{relative} Heffter arrays.
\begin{defi}\label{def:relative}
Let $v=2nk+t$ be a positive integer,
where $t$ divides $2nk$,  and
let $J$ be the subgroup of $\Z_{v}$ of order $t$.
 A  \emph{Heffter array $A$ over $\Z_{v}$ relative to $J$}, denoted by $\H_t(m,n; h,k)$, is an $m\times n$ array
 with elements in $\Z_{v}\setminus J$ such that:
\begin{itemize}
\item[$(\rm{a_2})$] each row contains $h$ filled cells and each column contains $k$ filled cells;
\item[$(\rm{b_2})$] for every $x \in \mathbb{Z}_v \setminus J$, either $x$ or $-x$ appears in the array;
\item[$(\rm{c_2})$] the elements in every row and column sum to $0$ in $\Z_v$.
\end{itemize}
\end{defi}
Analogously, one can define a weak relative Heffter array $\W\H_t(m,n;h,k)$, it is sufficient to replace condition $(\rm{b_2})$ of Definition \ref{def:relative}
with the following one:
\begin{itemize}
\item[$(\rm{b_3})$]  for every $x \in \mathbb{Z}_{v} \setminus J$, there is exactly one cell of $A$ whose element is one of the following:
$x,-x,\pm x,\mp x$, where the upper sign on $\pm$ or $\mp$ is the row sign  and
the lower sign is the column sign.
\end{itemize}
Clearly, if $t=1$ we find again the original concepts introduced in \cite{A}. For the square case we use the notation $\H_t(n;k)$ and $\W\H_t(n;k)$.

In details, \color{black} in Section \ref{sec:3} we determine some necessary conditions for the existence of a weak (relative) Heffter array
and then we present some non-existence results. Furthermore, in Propositions \ref{prop:t33} and \ref{prop:t43}
we show that for some values of $t,n$ and $k$ there exists an $\H_t(n;k)$ while no strictly weak $\W\H_t(n;k)$ exists.
In Section \ref{sec:4} we construct an infinite class of strictly weak relative Heffter arrays for some values of the parameters
for which the existence of the corresponding relative Heffter array was left open in \cite{CMPP}. In Sections \ref{sec:5} and \ref{sec:6r}, we investigate the relation between weak Heffter arrays and embeddings. 
Indeed while the connection between classical Heffter arrays and embeddings into orientable surfaces has been considered by several papers (see, for instance, \cite{CDY, CMPPHeffter, CPPBiembeddings, DM}), the one between weak Heffter arrays and embeddings has been only marginally investigated in \cite{A}. Here, in Section \ref{sec:5} we provide a formal definition of Archdeacon embedding into non-necessarily orientable surfaces. Then, in Section \ref{sec:6r}, we  present an infinite family of non-orientable embeddings of Archdeacon type.
In the last section we conclude with some remarks and we present an open problem.

\section{Connection between weak and classical Heffter arrays}\label{sec:2}
As suggested by the terminology, a weak Heffter array is an Heffter array in which a property is relaxed.
So we believe it is quite natural to ask if it is possible to construct a weak Heffter array starting from a classical one.
Clearly, given a Heffter array $A$ if we replace all the elements $a$ of $A$ with $\pm a$
we trivially get a weak Heffter array.
It is also easy to see that given a set $\mathcal{R}$ of rows of $A$ if we  replace each element $a$ of  $\mathcal{R}$  with $\mp a$  we get a weak Heffter array. Obviously a similar reasoning can be done on the columns.

In what follows, given an array $A$, we respectively denote by $R_i$ and $C_j$ the $i$-th row and the $j$-th column of $A$. Then, by \textit{support of $A$} we mean the set of the absolute values of the elements contained in $A$, that is $supp(A) = \{|a|: a \in A\}$.
We point out that in this section, with a little abuse of notation, we identify a row (column) of a (weak) $\H(m,n;h,k)$
with the $h$-subset ($k$-subset) of $\Z_{2nk+1}$ whose elements are those of the given row (column).

The following result shows a way to get a weak Heffter array, starting from a classical one, having few cells with different row and column signs.

\begin{prop}\label{prop:weak_from_class}
If there exists an $\H(m,n;h,k)$ with a row or a column containing a proper subset whose sum is zero (modulo
$2nk+1$), then there exists a $\W\H(m,n;h,k)$ whose number of cells containing an element with different row and column signs
is at most equal to \rm{max}$\left\{\left\lfloor\frac{h}{2}\right\rfloor, \left\lfloor\frac{k}{2}\right\rfloor\right\}$.
\end{prop}
\begin{proof}
Let $A$ be an array as in the statement.
It is not restrictive to reason on a row. So let $S$ be a proper subset of a row $R$ of $A$
summing up to zero. Since every row sums to zero also the elements of $R\setminus S$ sum to zero,
hence we can suppose that $|S|\leq \left\lfloor\frac{h}{2}\right\rfloor$.
 Now we replace each element in $S$ by $\mp s$, and
we leave unchanged all the other elements in $A$. Call $B$ the new array.
Clearly, $supp(A)=supp(B)$ and all the rows of $B$ different from $R$ are nothing but the rows of $A$,
hence their sums are zero.
About $R$, the sum is still zero since we have changed the signs of the elements of a subset of $R$ whose sum is zero.
Finally, note that the column signs are not changed.
So the columns of $B$ sum to $0$.
\end{proof}

\begin{ex}\label{ex:WeakFromClass}
Consider the following $\H(8;6)$, taken from Example 1.4 of \cite{ADDY}:
$$ \begin{array}{|r|r|r|r|r|r|r|r|} \hline
-1 & 5 & 2 & -7 & -9 & 10 & & \\ \hline
3 & -4 & -6 & 8 & 11 & -12 & & \\ \hline
& & -13 & 17 & 14 & -19 & -21 & 22 \\ \hline
& & 15 & -16 & -18 & 20 & 23 & -24 \\ \hline
-33 & 34 & & & -25 & 29 & 26 & -31 \\ \hline
35 & -36 & & & 27 & -28 & -30 & 32 \\ \hline
38 & -43 & -45 & 46 & & & -37 & 41 \\ \hline
-42 & 44 & 47 & -48 & & & 39 & -40 \\ \hline
\end{array}
$$
Note that, for instance, the proper subset $\{-1,-9,10\}$ of the first row sums to $0$. Then, by applying the proof of the previous proposition we have the following $\W\H(8;6)$:
$$ \begin{array}{|r|r|r|r|r|r|r|r|} \hline
\bf{\pm1}&5&2&-7&\bf{\pm9}&\bf{\mp10} & & \\ \hline
3 & -4 & -6 & 8 & 11 & -12 & & \\ \hline
& & -13 & 17 & 14 & -19 & -21 & 22 \\ \hline
& & 15 & -16 & -18 & 20 & 23 & -24 \\ \hline
-33 & 34 & & & -25 & 29 & 26 & -31 \\ \hline
35 & -36 & & & 27 & -28 & -30 & 32 \\ \hline
38 & -43 & -45 & 46 & & & -37 & 41 \\ \hline
-42 & 44 & 47 & -48 & & & 39 & -40 \\ \hline
\end{array}
$$
\end{ex}

We believe that it is natural to ask if some of the classical Heffter arrays have the property of Proposition \ref{prop:weak_from_class}.
First of all note that since a row (column) of an $\H(m,n;h,k)$ does not contain $0$ nor $2$-subsets of the form $\{x,-x\}$, if the property is satisfied by a row (column) then $h\geq 6$ ($k\geq 6$).
We recall that in all the constructions of square integer Heffter arrays, see \cite{ADDY,DW}, the authors obtain an $\H(n;k+4)$ starting from an $\H(n;k)$
by adding to each row and each column four elements having sum zero. So the condition required by Proposition \ref{prop:weak_from_class}
is trivially satisfied. Hence we get the following.
\begin{cor}
For every $n> k\geq 3$, there exists an integer $\W\H(n;k)$ with exactly $\ell$ cells containing an element with different row and column signs, where
\begin{itemize}
\item[(1)] $\ell=3$ if $k\equiv 3 \pmod 4$ and $n\equiv 0,1\pmod 4$;
\item[(2)] $\ell=4$ if $k\equiv 0 \pmod 4$;
\item[(3)] $\ell=4$ if $k\equiv 1 \pmod 4$, $k\neq 5$ and $n\equiv 0,3\pmod 4$;
\item[(4)] $\ell=4$ if $k\equiv 2 \pmod 4$, $k\neq 6$ and $n$ is even;
\end{itemize}
\end{cor}
\begin{proof}
In each case there exists an integer Heffter array having a row or a column containing a (not necessarily proper) subset 
of size $\ell$ whose sum is zero.
Hence the result follows by Proposition \ref{prop:weak_from_class}.
In the following table we summarize where the results on
classical Heffter arrays having the above property have been obtained. Note $n$ and $k$
represent congruence classes modulo $4$.
 \begin{center}
\begin{tabular} {|c||c|c|c|c|} \hline
$n\backslash k$&          0&                   1&          2&3            \\ \hline\hline
0              &Theorem 2.1 \cite{ADDY}& Theorem 4.3 \cite{DW}& Theorem 2.1 \cite{ADDY}& Theorem 3.11 \cite{ADDY} \\ \hline
1              &Corollary 2.4 \cite{ADDY}&               &    & Theorem 3.5 \cite{ADDY} \\ \hline
2              &Theorem 2.1 \cite{ADDY}&                & Theorem 2.1\cite{ADDY}&    \\ \hline
3              &Corollary 2.4 \cite{ADDY}& Theorem 3.3 \cite{DW}&      &      \\ \hline
\end{tabular}
\end{center}
\label{table}

\end{proof}

\begin{rem}
In the cases not considered in the previous table
 by the constructions it is not immediate to understand
if there exists a row or a column of the array containing a proper subset summing up to zero.
Hence one should check all the constructions into details, but this is not aim of this paper.
\end{rem}

\begin{rem}\label{rem:WeakFromClass}
Given a Heffter array $A$, let $\mathcal{S}$ be a set of rows and columns of $A$. For every row $R_i\in \mathcal{S}$, we replace each element $r \in R_i$ with $\mp r$ and for every column $C_j\in \mathcal{S}$, we replace each element $c \in C_j$ with $\pm c$. Note that, if an element $x$ belong to both a row $R_i$ and a column $C_j$ of $\mathcal{S}$, we replace it with $\pm(\mp x)=-x$. 
Then the result is a weak Heffter array. Clearly, the set $\mathcal{S}$ may contain no rows or no columns.
\end{rem}

\begin{ex}
Let $A$ be the $\H(8;6)$ of Example \ref{ex:WeakFromClass} and
set $\mathcal{S}=\{R_3,R_4,C_6\}$.
Following previous remark, starting from $A$ we get the following $\W\H(8;6)$.
$$ \begin{array}{|r|r|r|r|r|r|r|r|} \hline
-1 & 5 & 2 & -7 & -9 & \bf{\pm10} & & \\ \hline
3 & -4 & -6 & 8 & 11 & \bf{\mp12} & & \\ \hline
& & \bf{\pm13} & \bf{\mp17} & \bf{\mp14} & \bf{19} & \bf{\pm21} & \bf{\mp22} \\ \hline
& & \bf{\mp15} & \bf{\pm16} & \bf{\pm18} & \bf{-20} & \bf{\mp23} & \bf{\pm24} \\ \hline
-33 & 34 & & & -25 & \bf{\pm29} & 26 & -31 \\ \hline
35 & -36 & & & 27 & \bf{\mp28} & -30 & 32 \\ \hline
38 & -43 & -45 & 46 & & & -37 & 41 \\ \hline
-42 & 44 & 47 & -48 & & & 39 & -40 \\ \hline
\end{array}
$$
\end{ex}

It is really easy to see that one can apply on the same array Proposition \ref{prop:weak_from_class} and Remark \ref{rem:WeakFromClass};
for this reason we believe it is not necessary to write all the details, hence we only present this idea in the following example.

\begin{ex}
Let again $A$ be the $\H(8;6)$ of Example \ref{ex:WeakFromClass}.
Note that $S=\{11,14,-25\}$ is a subset of $C_5$ whose sum is zero.
Set $\mathcal{S}=\{R_7,C_6\}$.
Starting from $A$ we get the following $\W\H(8;6)$.
$$ \begin{array}{|r|r|r|r|r|r|r|r|} \hline
-1 & 5 & 2 & -7 & -9 & \bf{\pm10} & & \\ \hline
3 & -4 & -6 & 8 & \bf{\pm11} & \bf{\mp12} & & \\ \hline
& & -13 & 17 & \bf{\pm14} & \bf{\mp19} & -21 & 22 \\ \hline
& & 15 & -16 & -18 & \bf{\pm20} & 23 & -24 \\ \hline
-33 & 34 & & & \bf{\mp25} & \bf{\pm29} & 26 & -31 \\ \hline
35 & -36 & & & 27 & \bf{\mp28} & -30 & 32 \\ \hline
\bf{\mp38} & \bf{\pm43} & \bf{\pm45} & \bf{\mp46} & & & \bf{\pm37} & \bf{\mp41} \\ \hline
-42 & 44 & 47 & -48 & & & 39 & -40 \\ \hline
\end{array}
$$
\end{ex}

Clearly all the arrays constructed with one of the techniques illustrated in this section are, by definition, not strictly weak.

\section{Necessary conditions for weak Heffter arrays and non-existence results}\label{sec:3}

In this section we present some non-existence results on (strictly) weak Heffter arrays.
It is important to underline that
the existence of a strictly weak Heffter array is not always granted. In fact, it seems that, despite having an higher degree of freedom in the choice of the signs of the elements in the array, strictly weak Heffter arrays are as difficult to find as classical ones.

Furthermore, it may happen that there exists a Heffter array but not a strictly weak Heffter array with the same parameters,
as shown in the following remark.
\begin{rem}\label{rem:n34}
By Theorem \ref{thm:Heffter} an $\H(n;k)$ exists for any $n\geq k \geq 3$. On the other hand we have checked, with the aid of a computer,
that no strictly $\W\H(n;3)$ exists when $n= 3,4$.
\end{rem}
\begin{rem}\label{rem:weakwrtclassical}
We point out that if there is no weak Heffter array this means there is also no classical or strictly weak Heffter array with the same parameters.
While if  there is no \emph{strictly} weak Heffter array, this gives no information about the existence of a
classical or weak Heffter array with the same parameters.
\end{rem}

We start with a preliminary consideration.
\begin{prop}
A weak Heffter array $\W\H(m,n;h,k)$ has at least $3$ cells containing distinct row and column signs.
\end{prop}
\begin{proof}
In this proof given an array $A$ by $a_{ij}^R$ (respectively, $a_{ij}^C$) we denote the element of $A$
in the position $(i,j)$ with its row (respectively, column) sign.
Also, by $\sum A^R$ (respectively, $\sum A^C$) we mean the sum of all the elements in $A$ with their row (respectively, column) sign.
Set $I=\{(i,j) \mid a_{ij}^R\neq a_{ij}^C\}$.
If $A$ is a $\W\H(m,n;h,k)$, then the sum of all its elements is zero in $\Z_{2nk+1}$. Hence
$$\sum A^R=\sum A^C\equiv 0 \pmod {2nk+1},$$
which implies
$$2\sum_{(i,j)\in I}a_{ij}^R \equiv 0 \pmod {2nk+1}.$$
From this, it follows that $|I|\neq 1$ since $0\not \in A$ and $\Z_{2nk+1}$ does not contain the involution.
Also $|I|\neq 2$ since $A$ does not contain opposite elements and, again, the involution does not exist.
The thesis follows.
\end{proof}

In the remain of this section we propose some non-existence results for weak Heffter arrays and we compare the existence of classical and strictly weak Heffter arrays for the same given class of parameters.

\begin{rem}
A weak relative Heffter array cannot have exactly one cell with different row and column signs
since it does not contain $0$ and the involution (if it exists). On the other hand, a weak relative Heffter array may have exactly two cells with different row and column signs,
as in the following example, that is a $\W\H_{16}(4;4)$ (rielaboration of Example 1.4 of \cite{CMPP}):
\[
\begin{array}{|r |r|r |r|}\hline
\pm1 & -7 &-16 &22 \\  \hline
\pm23& 2  &-8 & -17\\ \hline
-13 & 19 & 4 &-10 \\ \hline
-11& -14& 20 & 5\\ \hline
\end{array}
\]

\end{rem}

To present our non-existence results we have to recall that the rows and the columns of a (weak) Heffter array
give two Heffter systems. Here we present a generalized definition in which we consider a non-trivial subgroup
of a cyclic group.

\begin{defi}
A \emph{Heffter system} $D_t(v;k)$ is a set of zero-sum $k$-subsets (called blocks) of $\Z_v$ such that for every $x\in \Z_v\setminus J$,
where $J$ is the subgroup of $\Z_v$ of order $t$, exactly once of each pair $\{x,-x\}$ is contained in exactly one block.
\end{defi}
If $t=1$ we find the classical concept of a Heffter system, see \cite{MR}.

It is clear that given a subgroup $J$ of $\Z_v$, the existence of a Heffter system on $\Z_v \setminus J$ is a necessary condition for the
existence of a (weak) Heffter array on $\Z_v$ relative to $J$. This fact leads us to generalise the statement of Lemma 3.3 of \cite{CMPP}, 
and to obtain a non-existence result of weak Heffter arrays for an infinite family of parameters.

\begin{lem}\label{lem:3n}
There is no Heffter system $D_{3n}(9n;3)$, for every $n\geq 3$.
\end{lem}
\begin{proof}
Assume that there exists such a Heffter system, denoted by $D$. It is easy to see that its support is
$\left\{1,2,3,\ldots, \floor{ \frac{9n}{2}}  \right\}\setminus \left\{3,6,9,\ldots,3\floor{ \frac{3n}{2}}\right\}$.

For every triple  $\{a_1,a_2,a_3\} \in D$, the sum of its elements must be zero in
$\Z_{9n}$, and since none of them can be equal to zero modulo $3$, they must belong to the same residue class modulo $3$.

Moreover, as for every element $a$ in $D$ its opposite $-a$ is not contained in $D$,
we may assume that every triple of $D$  contains elements all congruent to $1$ modulo $3$. In fact, if any given triple $\{a_1,a_2,a_3\} \in D$ is such that $a_i \equiv 2 \pmod{3}$, $i=1,2,3$, then one can construct another  Heffter system $D'$ by substituting $\{a_1,a_2,a_3\}$ with $\{-a_1,-a_2,-a_3\}$.

However, the elements in $\Z_{9n}\setminus J$ equal to $1$ modulo $3$ are:
\[
\begin{aligned}
&\left\{1,-2, 4,-5,7,-8,\ldots,\frac{9n-4}{2}, -\frac{9n-2}{2}\right\}\qquad \text{if $n$ is even,} 	\\
	&\left\{1,-2, 4,-5,7,-8,\ldots,-\frac{9n-3}{2}, \frac{9n-1}{2}\right\}\qquad \text{if $n$ is odd.}\\
\end{aligned}
\]
It is then evident that, in any case, these elements cannot sum to zero in $\Z_{9n}$, giving a contradiction.
\end{proof}

\begin{prop}\label{prop:3n}
There is no $\W\H_{3n}(n;3)$ for every $n\geq 3$.
\end{prop}
\begin{proof}
The result immediately follows from previous lemma.
\end{proof}
We point out that there are examples where there exists  exactly  one Heffter system with some given parameters.
Hence, also in these cases, it is not possible to construct a Heffter array, neither classical nor weak.
Finally, also the existence of at least two Heffter systems is not a sufficient condition for the existence of a Heffter array, see
Proposition \ref{prop:t43}.

Now, as a consequence of Remark \ref{rem:n34}, we believe that it is interesting to study the existence problem of a $\W\H_t(n;3)$
when $n=3,4$ for every admissible $t$. Firstly, we present the following result whose proof relies on the fact that, if
we are working in $\Z$ or in $\Z_v$ with $v$ even, the array must contain an even number of odd numbers.
The proof is omitted since it is similar to that of Proposition 3.1 of \cite{CMPP}.

\begin{prop}\label{prop:necc}
Suppose that there exists an integer $\W\H_t(n;k)$, or a non-integer   $\W\H_t(n;k)$ with $t$ even.
\begin{itemize}
\item[(1)] If $t$ divides $nk$, then
$$nk\equiv 0 \pmod 4 \quad \textrm{ or } \quad nk\equiv -t \equiv \pm 1\pmod 4.$$
\item[(2)] If $t=2nk$, then $k$ must be even.
\item[(3)] If $t\neq 2nk$ does not divide $nk$, then
$$t+2nk\equiv 0 \pmod 8.$$
\end{itemize}
\end{prop}

\begin{prop}\label{prop:t33}
There exists
\begin{itemize}
\item[(1)] an $\H_t(3;3)$ if and only if $t=1,3,6$;
\item[(2)] a strictly weak $\W\H_t(3;3)$ if and only if $t=6$.
\end{itemize}
\end{prop}
\begin{proof}
Since $t$ has to divide $2nk=18$, it may assume only values $1,2,3,6,9,18$.
We already know that a $\W\H_t(3;3)$ does not exist for $t=2$, see Proposition \ref{prop:necc}(3),
for $t=9$, see Proposition \ref{prop:3n} and for $t=18$, see Proposition \ref{prop:necc}(2).
Hence for these values of the parameters not a classical or a strictly weak Heffter array can exist (see Remark \ref{rem:weakwrtclassical}).

(1) The existence of an $\H_t(3;3)$ is known for $t=1$, see Theorem \ref{thm:Heffter}, and $t=3$, see Theorem 1.5 of \cite{CMPP}.
The following is an $\H_6(3;3)$:
$$
\begin{array}{|c|c|c|} \hline
1 & 	2 &   	-3   	\\ \hline
5 &    9   & 10 \\ \hline
-6 	  &  -11 & -7 \\ \hline
\end{array}$$

(2) For $t=1$ a strictly $\W\H_t(3;3)$ does not exist by Remark \ref{rem:n34}.
For $t=3$ there are only the following four Heffter systems $D_3(21;3)$:
\begin{eqnarray}
\nonumber \mathcal{D}_1 &=&\left\{  \{1,2,-3 \},  \{4,8,9 \},   \{5,6,10 \}  \right\}, \cr
\nonumber \mathcal{D}_2 &=& \left\{  \{1,4,-5 \},  \{2,8,-10 \},  \{3,6,-9 \}   \right\}, \cr
\nonumber \mathcal{D}_3 &=& \left\{  \{1,8,-9 \},  \{2,3,-5 \},  \{4,6,-10 \} \right\}, \cr
\nonumber \mathcal{D}_4 &=& \left\{  \{1,9,-10 \},  \{2,4,-6\},  \{3,5,-8 \}\right\}.
\end{eqnarray}
We have checked, with the aid of a computer, that using them it is not possible to construct a strictly weak Heffter array.
Finally, the following is a strictly $\W\H_6(3;3)$:
$$
\begin{array}{|c|c|c|} \hline
1 & 	2 &   	-3   	\\ \hline
5 &   \pm 9   & \pm10 \\ \hline
-6 	  &  \mp7 & -11 \\ \hline
\end{array}$$
\end{proof}

\begin{prop}\label{prop:t43}
There exists 
\begin{itemize}
\item[(1)] an $\H_t(4;3)$ if and only if $t=1,2,3,4,6$;
\item[(2)] a strictly weak $\W\H_t(4;3)$ if and only if $t=2,4$.
\end{itemize}
\end{prop}
\begin{proof}
Since $t$ has to divide $2nk=24$, it may assume only values $1,2,3,4,6,8,12,24$.
We already know that a $\W\H_t(4;3)$ does not exist
for $t=12$, see Proposition \ref{prop:3n} and for $t=24$, see Proposition \ref{prop:necc}(2).
Furthermore for $t=8$ there exists exactly one Heffter system $D_8(32;3)$, that is
\[
\mathcal{D} = \left\{ \{3, 10, -13 \}, \{6,5,-11 \}, \{2,7,-9 \}, \{1,14,-15 \}  \right\}.
\]
Hence for all these values of the parameters neither a classical nor a strictly weak Heffter array can exist.

(1) The existence of an $\H_t(4;3)$ is known for $t=1$, see Theorem \ref{thm:Heffter}, and $t=3$, see Theorem 1.5 of \cite{CMPP}.
The following are an $\H_{2}(4;3)$, an $\H_{4}(4;3)$ and an $\H_6(4;3)$:
$$
\begin{array}{|c|c|c|c|} \hline
-1 & 9 &  -8 &  	\\ \hline
12 &  -2	& 	   & 	-10 \\ \hline
  & 	-7  &   3 & 4 \\ \hline
		-11 &  & 5   & 6 \\ \hline
\end{array}
\qquad
\begin{array}{|c|c|c|c|} \hline
-1 & 	-5 &   	6 &  	\\ \hline
-10 &  2	& 	   & 	8 \\ \hline
 11  & 	  &    13 & 4 \\ \hline
     &     3&    9     &-12 \\ \hline
\end{array}
\qquad
\begin{array}{|c|c|c|c|} \hline
1 & 	-7 & 		6 &  	\\ \hline
-9 &  		& 	  11 & 	-2 \\ \hline
   & 	 3 &    13 & 14 \\ \hline
8 &     4&         &-12 \\ \hline
\end{array}
$$

(2) For $t=1$ a strictly $\W\H_t(4;3)$ does not exist by Remark \ref{rem:n34}.
The following are a strictly $\W\H_2(4;3)$ and a strictly $\W\H_4(4;3)$ :
$$
\begin{array}{|c|c|c|c|} \hline
1 & 	2 &  -3 &  	\\ \hline
4 &  	& 	\mp10   & 6 \\ \hline
-5  & 	\mp9  &  & \mp12 \\ \hline
 & -11 & -7  & \mp8 \\ \hline
\end{array}
\quad\quad
\begin{array}{|c|c|c|c|} \hline
1 & 	-4 &   	3 &  	\\ \hline
 &  10	& 	\mp2   & 	-8 \\ \hline
 \mp12  & 	  &  -5 & -11 \\ \hline
 -13 & -6 &     &-9 \\ \hline
\end{array}$$
For $t=3,6$ there are respectively $9$ and $10$ Heffter systems $D_t(24+t;3)$, and we have checked using a computer
that it is not possible to construct a strictly  $\W\H_t(4;3)$ starting from them.
\end{proof}

\section{An infinite class of strictly weak Heffter arrays}\label{sec:4}
In this section we construct an infinite family of strictly weak Heffter arrays for a class of parameters for which,
at the moment, the existence of a Heffter array is unknown.

In \cite{CMPP} the authors, after having introduced the concept of \emph{relative} Heffter arrays,
have obtained the following almost complete result for the integer square case with $t=k$.

\begin{thm}\label{thm:esistenza}
Let $3\leq k\leq n$ with $k\neq 5$.
  There exists an integer $\H_k(n;k)$ if and only if one of the following holds:
  \begin{itemize}
    \item[(1)] $k$ is odd and $n\equiv 0,3\pmod 4$;
    \item[(2)] $k\equiv 2\pmod 4$ and $n$ is even;
    \item[(3)] $k\equiv 0\pmod 4$.
  \end{itemize}
Furthermore, there exists an integer $\H_5(n;5)$ if
$n\equiv 3\pmod 4$ and it does not exist if $n\equiv 1,2\pmod 4$.
\end{thm}
Note that for $k=5$  the existence problem of integer relative Heffter arrays $\H_5(n;5)$ has been proved only for
$n\equiv 3\pmod 4$, leaving the case $n\equiv 0 \pmod 4$ open. For this class in \cite{CMPP} there are only two examples for $n=8$
and $n=16$.
Here we focus on this open case and we prove that there exists a strictly weak integer $\W\H_5(n;5)$ for every $n\equiv 0\pmod 4$, with $n\geq 12$.

The construction we are going to present is based on filling in the cells of a set of diagonals.
We recall that, given a square array of order $n$, for $i=1,\ldots,n$, the $i$-th diagonal is so defined
$D_i=\{(i,1),(i+1,2),\ldots, (i-1,n)\}$. All the arithmetic on row and column indices is performed modulo $n$, where the set of reduced
residues is $\{1,2,\ldots,n\}$. The diagonals $D_i,D_{i+1},\ldots, D_{i+k-1}$ are $k$ consecutive diagonals.
Given $n\geq k\geq 1$, a partially filled array $A$ of order $n$ is $k$-\emph{diagonal} if its non empty cells are exactly those of $k$ diagonals.
Furthermore, if these diagonals are consecutive $A$ is said to be \emph{cyclically} $k$-\emph{diagonal}.

Many of the known square Heffter arrays have a diagonal structure which has been shown to be useful for recursive constructions
and for the applications to biembeddings. In addition,  this diagonal structure has been extremely useful in computer searches for Heffter arrays.
In \cite{DW} the authors introduced the following smart notation to describe a diagonal array; we point out that this notation has been used in several subsequent papers including  \cite{CMPP,CPPBiembeddings}.
This procedure for filling a sequence of cells on a diagonal is termed {\em diag} and it has six parameters, as follows.

Let $A$ be an $n \times n$ p.f. array; then, the procedure $diag(r,c,s,\Delta_1,\Delta_2,\ell)$ fills the entries
\[
A[r+i\Delta_1,c+i\Delta_1]=s+i\Delta_2\qquad \textrm{for}\ i\in[0,\ell-1].
\]
The parameters used in the $diag$ procedure have the following meaning:
\begin{itemize}
  \item $r$ denotes the starting row,
  \item $c$ denotes the starting column,
  \item $s$ denotes the entry $A[r,c]$,
  \item $\Delta_1$ denotes the increasing value of the row and column at each step,
  \item $\Delta_2$ denotes how much the entry is changed at each step,
  \item $\ell$ is the length of the chain.
\end{itemize}
In the following given two positive integers $a,b$ with $a \leq b$, by $[a,b]$
we denote the set $\{a,a+1,\ldots,b\}$.
We are now ready to present the main existence result of this section.

We remark that in Example \ref{ex:W5} we follow step by step the construction illustrated in the proof
of the following theorem. Hopefully it can help the reader to understand the idea developed in the proof.

\begin{thm}\label{prop:existence}
There exists a strictly weak integer $\W\H_5(n;5)$ for every $n \equiv 0 \pmod{4}$ with $n\geq 12$.
\end{thm}
\begin{proof}
We begin by considering the integer cyclically $3$-diagonal $\H_3(n;3)$ with $n\equiv 0\pmod 4$ constructed in Proposition 5.3 of \cite{CMPP},
so let $A$ be the $n \times n$ array built using the following procedures labeled $\texttt{A}$ to $\texttt{J}$:
$$\begin{array}{lcl}
\texttt{A}:\; diag\left(2,2,1,1,1,\frac{n-4}{2}\right); & \hfill &
\texttt{B}:\; diag\left(\frac{n+6}{2},\frac{n+6}{2},-\frac{n+4}{2},1, -1,\frac{n-4}{2}\right);\\[3pt]
\texttt{C}:\; diag\left(2,1,-\frac{5n+4}{2},2,-1,\frac{n}{4}\right); & &
\texttt{D}:\; diag\left(3,2,-\frac{3n+2}{2},2,-1,\frac{n-4}{4}\right);\\[3pt]
\texttt{E}:\; diag\left(1,2,\frac{3n}{2},2,-1,\frac{n}{4}\right); &&
\texttt{F}:\; diag\left(2,3,\frac{5n+2}{2},2,-1,\frac{n-4}{4}\right);\\[3pt]
\texttt{G}:\; diag\left(\frac{n+6}{2},\frac{n+4}{2},-\frac{5n}{4},2,1,\frac{n}{4}\right); &&
\texttt{H}:\; diag\left(\frac{n+8}{2},\frac{n+6}{2},-\frac{9n}{4},2,1,\frac{n-4}{4}\right);\\[3pt]
\texttt{I}:\; diag\left(\frac{n+4}{2},\frac{n+6}{2},\frac{11n+8}{4},2,1,\frac{n}{4}\right); &&
\texttt{J}:\; diag\left(\frac{n+6}{2},\frac{n+8}{2},\frac{7n+8}{4},2,1,\frac{n-4}{4}\right).
\end{array}$$
We also fill the following cells of $A$ in an \textit{ad hoc} manner:
$$\begin{array}{lclcl}
A\left[1,1\right]=-\frac{n-2}{2}, & \quad  &
A\left[\frac{n}{2},\frac{n}{2}\right]=n, & \quad &
A\left[\frac{n}{2},\frac{n+2}{2}\right]=\frac{7n+4}{4}, \\[3pt]
A\left[\frac{n+2}{2},\frac{n}{2}\right]=-\frac{9n+4}{4},& &
A\left[\frac{n+2}{2},\frac{n+2}{2}\right]=\frac{n+2}{2},& &
A\left[\frac{n+2}{2},\frac{n+4}{2}\right]=\frac{7n}{4}, \\[3pt]
A\left[\frac{n+4}{2},\frac{n+2}{2}\right]=-\frac{9n+8}{4}, &&
A\left[\frac{n+4}{2},\frac{n+4}{2}\right]=-\frac{n}{2}.
  \end{array}$$	

Note that the filled cells of $A$ are exactly those of the diagonals $D_1$, $D_2$ and $D_n$. Also, it is easy to see that:
\[
\begin{aligned}
	&\text{\emph{supp}}(D_1) = [1,n] \\
	&\text{\emph{supp}}(D_2 \cup D_n) = [n+1,3n+1] \setminus\{2n+1\}.
	\end{aligned}
\]
Consider now the matrix $B$ obtained from $A$ by adding $4n+2$ to the positive elements of $D_1$
and $-(4n+2)$ to the negative elements of $D_1$. Since we have only changed the elements in the main diagonal of $A$,
it follows that the total sum of  the $i$-th row of $B$ is equal to the total sum of the $i$-th column of $B$, for every $i \in [1,n]$.
In particular, their sum is $4n+2$ if $i \in \left[2, \frac{n}{2}+1\right]$, and $-(4n+2)$ if $i \in \{1\} \cup \left[\frac{n}{2}+2,n\right]$.
Clearly now ${\emph{supp}}(D_1) = [4n+3,5n+2]$, while the support of $D_2$ and $D_n$ is unchanged.
Hence
$$supp(B)=[n+1,2n]\cup[2n+2,3n+1]\cup[4n+3,5n+2].$$

Consider the following arrays:
\[
M_1 =  \begin{bmatrix}
 \pm 1 & 4n+1 \\
 -4n 	 & \mp 2 \\
\end{bmatrix},
\qquad M_2 = \begin{bmatrix}
 \mp \left(\frac{n}{2}+1\right) & -\frac{7n}{2} \\
 \frac{7n}{2}+1 	 & \pm \left(\frac{n}{2}+2\right) \\
\end{bmatrix}
\]
and for any $ i \in \left\{3,5,7,\dotsc, \frac{n}{2}-1\right\}\cup \left\{\frac{n}{2}+3, \frac{n}{2}+5,\dotsc, n-1\right\}$, define $M_i$ to be:
\[
M_i = \begin{bmatrix}
i+1 & 4n+1-i 	\\
4n+2-i & i 	  	\\
\end{bmatrix}.
\]
Set $I=\{1,2\}\cup \left\{3,5,7,\dotsc, \frac{n}{2}-1\right\}\cup \left\{\frac{n}{2}+3, \frac{n}{2}+5,\dotsc, n-1\right\}$.
It can then be seen that $$supp\left( \bigcup_{i \in I} M_i \right) = [1,n] \cup [3n+2,4n+1].$$

Now starting from $B$ we construct a new array $C$ by adding to $B$ the arrays  $M_i$ in the following positions.
The elements of the array $M_1$ are inserted in the positions $(1,\frac{n}{2}+1)$, $(1,\frac{n}{2}+2)$, $(2,\frac{n}{2}+1)$,  and
$(2,\frac{n}{2}+2)$. Those of the array $M_2$ fill the positions $(\frac{n}{2}+1,1)$, $(\frac{n}{2}+1,2)$, $(\frac{n}{2}+2,1)$, and
$(\frac{n}{2}+2,2)$.
Then, for all $i \in \{3,5,7,\dotsc, \frac{n}{2}-1\}$ we insert the elements of $\mp M_i$ in the positions
$(i,\frac{n}{2}+i)$, $(i,\frac{n}{2}+i+1)$, $(i+1,\frac{n}{2}+i)$, $(i+1,\frac{n}{2}+i+1)$.
Finally, for all $i \in \{\frac{n}{2}+3, \dotsc, n-1\}$ the elements of the array $\pm M_i^T$,
where by $M_i^T$ we mean the transpose of $M_i$, are inserted in the positions
$(i,i-\frac{n}{2})$, $(i,i-\frac{n}{2}+1)$, $(i+1,i-\frac{n}{2})$, $(i+1,i-\frac{n}{2}+1)$.

Note that, in such a way, we added two filled cells in each row and each column. Hence
every row and every column of $C$ has exactly $5$ filled cells.
Also note that
$$supp(C)=supp(B)\cup supp \left(\cup_{i\in I}M_i\right)=[1,5n+2]\setminus \{2n+1,4n+2\}.$$

Now, notice that the ordered list of the sums of the rows of this new array $C$, that is the same of the sum of its columns, is:
\begin{equation} \label{eq:somme}
(\underbrace{0,0,\dotsc,0}_{\text{$n/2$ elements}}, +1,+1, \underbrace{+1,-1,+1,-1,\dotsc, +1,-1}_{\text{$n/2-2$ elements}}).
\end{equation}

To conclude the construction, we need to exchange some elements belonging to the main diagonal of $C$.
In particular, for every $i \in \{\frac{n}{2}+3, \frac{n}{2}+5,\dotsc, n-1\}$ we exchange the element $C[i,i]$
with $C[i+1,i+1]$. Finally, in the position $(\frac{n}{2}+1,\frac{n}{2}+1)$ we put $-C[\frac{n}{2}+2,\frac{n}{2}+2]$
and in the position $(\frac{n}{2}+2,\frac{n}{2}+2)$ we put $-C[\frac{n}{2}+1,\frac{n}{2}+1]$.

Clearly, the support of this new array is nothing but the support of $C$.
The property that now the total sum  of the rows and the columns is zero follows by combining these last assignments with Equation (\ref{eq:somme}).
\end{proof}

\begin{ex}\label{ex:W5}
In this example we follow step by step  the proof of Theorem \ref{prop:existence} to construct a strictly $\W\H_5(12;5)$.

Firstly, let $A$ be the $\H_3(12;3)$ obtained by applying the construction of Proposition 5.3 of \cite{CMPP} (for convenience we have highlighted its subdivision into $2 \times 2$ blocks):
$$\begin{array}{ |r |r| |r |r| |r |r| |r |r| |r |r| |r |r|}\hline
-5 & 18 &  &  &  &  &  &  &  &  &  & -13\\ \hline
-32 & 1 & 31 &  &  &  &  &  &  &  &  & \\ \hline\hline
 & -19 & 2 & 17 &  &  &  &  &  &  &  & \\ \hline
 &  & -33 & 3 & 30 &  &  &  &  &  &  & \\ \hline\hline
 &  &  & -20 & 4 & 16 &  &  &  &  &  & \\ \hline
 &  &  &  & -34 & 12 & 22 &  &  &  &  & \\ \hline\hline
 &  &  &  &  & -28 & 7 & 21 &  &  &  & \\ \hline
 &  &  &  &  &  & -29 & -6 & 35 &  &  & \\ \hline\hline
 &  &  &  &  &  &  & -15 & -8 & 23 &  & \\ \hline
 &  &  &  &  &  &  &  & -27 & -9 & 36 & \\ \hline\hline
 &  &  &  &  &  &  &  &  & -14 & -10 & 24\\ \hline
37 &  &  &  &  &  &  &  &  &  & -26 & -11\\ \hline
\end{array}$$
We now construct the array $B$ obtained from $A$ by adding  $4n+2 = 50$
to the positive elements of $D_1$ and $-50$ to the negative ones:
$$\begin{array}{ |r |r| |r |r| |r |r| |r |r| |r |r| |r |r|}\hline
\mathbf{-55} & 18 &  &  &  &  &  &  &  &  &  & -13\\ \hline
-32 & \mathbf{51} & 31 &  &  &  &  &  &  &  &  & \\ \hline\hline
 & -19 & \mathbf{52} & 17 &  &  &  &  &  &  &  & \\ \hline
 &  & -33 & \mathbf{53} & 30 &  &  &  &  &  &  & \\ \hline\hline
 &  &  & -20 & \mathbf{54} & 16 &  &  &  &  &  & \\ \hline
 &  &  &  & -34 & \mathbf{62} & 22 &  &  &  &  & \\ \hline\hline
 &  &  &  &  & -28 & \mathbf{57} & 21 &  &  &  & \\ \hline
 &  &  &  &  &  & -29 & \mathbf{-56} & 35 &  &  & \\ \hline\hline
 &  &  &  &  &  &  & -15 & \mathbf{-58} & 23 &  & \\ \hline
 &  &  &  &  &  &  &  & -27 & \mathbf{-59} & 36 & \\ \hline\hline
 &  &  &  &  &  &  &  &  & -14 & \mathbf{-60} & 24\\ \hline
37 &  &  &  &  &  &  &  &  &  & -26 & \mathbf{-61}\\ \hline
\end{array}$$
Consider the arrays $M_i$ for $i=\{1,2\}\cup \{3,5\}\cup \{9,11\}$:

\[
\begin{aligned}
&M_1 =  \begin{bmatrix}
 \pm 1 & 49 \\
 -48 	 & \mp 2 \\
\end{bmatrix},
\qquad &M_2 = \begin{bmatrix}
 \mp 7 & -42\\
 43 	 & \pm 8 \\
\end{bmatrix},
\qquad &M_3 = \begin{bmatrix}
4& 46 	\\
47 & 3	  	\\
\end{bmatrix}, \\
 &M_5 = \begin{bmatrix}
6& 44 	\\
45 & 5	  	\\
\end{bmatrix}, 
 \qquad &M_9 = \begin{bmatrix}
10& 40 	\\
41 & 9	  	\\
\end{bmatrix},
\qquad &M_{11} = \begin{bmatrix}
12& 38 	\\
39 & 11  	\\
\end{bmatrix}. \\
\end{aligned}
\]
We then insert the elements of these arrays in $B$ obtaining the new array $C$:
$$\begin{array}{|r |r| |r |r| |r |r| |r |r| |r |r| |r |r|}\hline
-55 & 18 &  &  &  &  & 		\bf{\pm1} & \bf{49} &  &  &  & -13\\ \hline
-32 & 51 & 31 &  &  &  & \bf{-48} &\bf{\mp 2} &  &  &  & \\ \hline\hline
 & -19 & 52 & 17 &  &  &  &  & \bf{\mp4} & \bf{\mp46} &  & \\ \hline
 &  & -33 & 53 & 30 &  &  &  & \bf{\mp47} & \bf{\mp3} &  & \\ \hline\hline
 &  &  & -20 & 54 & 16 &  &  &  &  & \bf{\mp6} & \bf{\mp44}\\ \hline
 &  &  &  & -34 & 62 & 22 &  &  &  & \bf{\mp45} & \bf{\mp 5}\\ \hline\hline
\bf{\mp7} & \bf{-42} &  &  &  & -28 & 57 & 21 &  &  &  & \\ \hline
\bf{43} & \bf{\pm8} &  &  &  &  & -29 & -56 & 35 &  &  & \\ \hline\hline
 &  & \bf{\pm10} &\bf{\pm 41} &  &  &  & -15 & -58 & 23 &  & \\ \hline
 &  & \bf{\pm40}& \bf{\pm9} &  &  &  &  & -27 & -59 & 36 & \\ \hline\hline
 &  &  &  & \bf{\pm12} & \bf{\pm39} &  &  &  & -14 & -60 & 24\\ \hline
37 &  &  &  & \bf{\pm38} & \bf{\pm11} &  &  &  &  & -26 & -61\\ \hline
\end{array}$$
Notice that the ordered list of the total sum of the rows and that of the columns of $C$ is:
\[
(0,0,0,0,0,0,+1,+1,+1,-1,+1,-1).
\]
Hence, we need to apply the final readjustment of some elements of the main diagonal of $C$ to gain the zero sum on every row and column (see the elements in bold):
$$\begin{array}{|r |r| |r |r| |r |r| |r |r| |r |r| |r |r|}\hline
-55 & 18 &  &  &  &  & 		\pm1 & 49 &  &  &  & -13\\ \hline
-32 & 51 & 31 &  &  &  & -48 &\mp 2 &  &  &  & \\ \hline\hline
 & -19 & 52 & 17 &  &  &  &  & \mp4 & \mp46 &  & \\ \hline
 &  & -33 & 53 & 30 &  &  &  & \mp47 & \mp3 &  & \\ \hline\hline
 &  &  & -20 & 54 & 16 &  &  &  &  & \mp6 & \mp44\\ \hline
 &  &  &  & -34 & 62 & 22 &  &  &  & \mp45 & \mp 5\\ \hline\hline
\mp7 & -42 &  &  &  & -28 & \mathbf{56} & 21 &  &  &  & \\ \hline
43 & \pm8 &  &  &  &  & -29 & \mathbf{-57} & 35 &  &  & \\ \hline\hline
 &  & \pm10 &\pm 41 &  &  &  & -15 & \mathbf{-59} & 23 &  & \\ \hline
 &  & \pm40 & \pm9 &  &  &  &  & -27 & \mathbf{-58} & 36 & \\ \hline\hline
 &  &  &  & \pm12 & \pm39 &  &  &  & -14 & \mathbf{-61} & 24\\ \hline
37 &  &  &  & \pm38 & \pm11 &  &  &  &  & -26 & \mathbf{-60}\\ \hline
\end{array}$$
It can now be seen that the obtained array is a strictly $\W\H_5(12;5)$.
\end{ex}

\section{The Archdeacon Embedding}
\label{sec:5}
This section focuses mainly on the connection between partially filled arrays and embeddings. To explain this link, we first recall some basic definitions, see \cite{Moh, MT}.
\begin{defi}
Given a graph $\Gamma$ and a surface $\Sigma$, an \emph{embedding} of $\Gamma$ in $\Sigma$ is a continuous injective mapping $\psi: \Gamma \rightarrow \Sigma$, where $\Gamma$ is viewed with the usual topology as $1$-dimensional simplicial complex.
\end{defi}
The connected components of $\Sigma \setminus \psi(\Gamma)$ are said $\psi$-\emph{faces}. Also, with abuse of notation, we say that a circuit $F$ of $\Gamma$ is a face (induced by the embedding $\psi$) if $\psi(F)$ is the boundary of a $\psi$-face. Then, if each $\psi$-face is homeomorphic to an open disc, the embedding $\psi$ is called \emph{cellular}.

Following \cite{GT}, we provide an equivalent, but purely combinatorial, definition of graph embedding into a surface.
Here, we denote by $D(\Gamma)$ the set of the oriented edges of the graph $\Gamma$ and, given a vertex $x$ of $\Gamma$, by $N(\Gamma,x)$ the neighborhood of $x$ in $\Gamma$.
\begin{defi}\label{DefEmbeddings}
Let $\Gamma$ be a connected multigraph. A \emph{combinatorial embedding} of $\Gamma$ is a triple $\Pi=(\Gamma,\epsilon,\rho)$ where $\epsilon:E(\Gamma)\rightarrow \{-1,1\}$ and $\rho: D(\Gamma)\rightarrow D(\Gamma)$ satisfies the following properties:
\begin{itemize}
\item[(a)] for any $y\in N(\Gamma,x)$, there exists $y'\in N(\Gamma,x)$ such that $\rho(x,y)=(x,y')$,
\item[(b)] we define $\rho_x$ as the permutation of $N(\Gamma,x)$ such that, given $y\in N(\Gamma,x)$, $\rho(x,y)=(x,\rho_x(y))$. Then the permutation $\rho_x$ is a cycle of order $|N(\Gamma,x)|$.
\end{itemize}
If properties (a) and (b) hold, the map $\rho$ is said to be a \emph{rotation} of $\Gamma$.
\end{defi}
Then, as reported in \cite{GT}, a combinatorial embedding $\Pi=(\Gamma,\epsilon,\rho)$ is equivalent to a cellular embedding $\psi$ of $\Gamma$ into a surface $\Sigma$.
Now we revisit the definition of the Archdeacon embedding in the case of weak Heffter arrays. We first introduce some notation.
Given a partially filled array $A$, we denote by $\E(A)$, $\E(R_i)$, $\E(C_j)$ the list of the elements of the filled cells of $A$, of the $i$-th row and of the $j$-th column, respectively.
By $\omega_{R_i}$ and $\omega_{C_j}$ we mean an ordering of, respectively, $\E(R_i)$ and $\E(C_j)$. We then define by $\omega_r=\omega_{R_1}\circ \cdots \circ \omega_{R_m}$ the ordering for the rows and by $\omega_c=\omega_{C_1}\circ \cdots \circ \omega_{C_n}$ the ordering for the columns.
\begin{ex}
Let $A$ be the weak Heffter array $\W\H(3,4)$ of Example \ref{ex:WeakA}:
$$\begin{array}{|c|c|c|c|} \hline
1 & -7 & -6 & 12 \\ \hline
2 & -4 & \pm 10 & \mp 8 \\ \hline
-3 & \mp 11 & \pm 9 & 5 \\ \hline
\end{array}$$

We have that $$\E(A)=\{1,-7,-6,12,2,-4,\pm 10,\mp 8,-3,\mp 11,\pm 9,5\}.$$
Recalling that the rows (resp. columns) are obtained by considering the upper (resp. lower) signs, we have that
\begin{center}
$\begin{array}{lcl}
\E(R_1)=\{1,-7,-6,12\}; & \quad &\E(C_1)=\{1,2,-3\};\\
\E(R_2)=\{2,-4,+10,-8\}; & \quad &\E(C_2)=\{-7,-4,+11\};\\
\E(R_3)=\{-3,-11,+9,5\};& \quad &\E(C_3)=\{-6,-10,-9\};\\
 & \quad &\E(C_4)=\{12,+8,5\}.
\end{array}$
\end{center}

 Here if we consider the natural orderings, from left to right for the rows and from top to bottom for the columns, we have that
\begin{center}
$\begin{array}{lcl}
\omega_{R_1}=(1,-7,-6,12); & \quad &\omega_{C_1}=(1,2,-3);\\
\omega_{R_2}=(2,-4,+10,-8); & \quad &\omega_{C_2}=(-7,-4,+11);\\
\omega_{R_3}=(-3,-11,+9,5);& \quad &\omega_{C_3}=(-6,-10,-9);\\
 & \quad &\omega_{C_4}=(12,+8,5).
\end{array}$
\end{center}
It follows that
$$ \omega_{r}=(1,-7,-6,12)(2,-4,+10,-8)(-3,-11,+9,5)$$
and
$$\omega_c=(1,2,-3)(-7,-4,+11)(-6,-10,-9)(12,+8,5).$$
\end{ex}

Let $A$ be a $\W\H(m,n;h,k)$, we denote by $\Theta(A)$ the $m\times n$ array obtained by deleting from $A$ the elements of type $\pm x$ or $\mp x$. Then we denote by $\Omega(A)$ the complement of $\Theta(A)$ in $A$. We note that
$$\mathbb{Z}_{2nk+1}\setminus \{0\}=\pm\E(\Theta(A)) \cup \E(\Omega(A)).$$
Therefore, we can define the map $\lambda: \Z_{2nk+1}\setminus\{0\} \rightarrow \{1,-1\}$ such that
$$\lambda(a)=\begin{cases}1 \mbox{ if }a\in \pm \E(\Theta(A));\\
-1 \mbox{ if }a\in \E(\Omega(A)).
\end{cases}
$$
\begin{ex}
We consider again the array $A$ of Example \ref{ex:WeakA}. Here we have that:
\begin{center}
$\Theta(A)=\begin{array}{|c|c|c|c|} \hline
1 & -7 & -6 & 12 \\ \hline
2 & -4 & & \\ \hline
-3 & & & 5 \\ \hline
\end{array}
\ \ \ 
\Omega(A)=\begin{array}{|c|c|c|c|} \hline
& & & \\ \hline
& & \pm 10 & \mp 8 \\ \hline
& \mp 11 & \pm 9 & \\ \hline
\end{array}
$
\end{center}
Note that, since $A$ is a $\W\H(3,4)$,
$$\mathbb{Z}_{25}\setminus \{0\}=\pm\E(\Theta(A)) \cup \E(\Omega(A)).$$
Here we have that
$\lambda(x)=1$ whenever $x\in \pm\{1,-7,-6,12,2,-4,-3,5\}$ and $-1$ otherwise.
\end{ex}
\begin{defi}\label{sequence}
Let $A$ be a $\W\H(m,n;h,k)$ and let us consider the ordering $\omega_r$ for the rows and $\omega_c$ for the columns. Note that we can assume, without loss of generality, that the cell $(1,1)$ is filled. We denote by $a_1$ the element contained in the cell $(1,1)$, considered with its row sign so that $\omega_r(a_1)$ and $\omega_r^{-1}(a_1)$ are well-defined. Then, we define recursively
\begin{eqnarray}\label{ArchEmb}
a_{i+1}&=&\begin{cases}
-\omega_c^{\mu_{i}}(\lambda(a_i)a_i)\mbox{ if  $i$ is odd};\\
\omega_r^{\mu_{i}}(-\lambda(a_i) a_i)\mbox{ if $i$ is even}\\
\end{cases}
\end{eqnarray}
 where $\mu_1=\lambda(a_1)$ and
\begin{eqnarray}\label{Memoria}
\mu_{i+1}&=&\begin{cases}
\mu_i\mbox{ if } \lambda(a_{i+1})=1;\\
-\mu_i\mbox{ if } \lambda(a_{i+1})=-1.
\end{cases}
\end{eqnarray}
From now on we will use the notation $(b_i)$ to indicate the sequence of elements $(b_1,b_2,\cdots)$.
In this discussion we prefer to avoid introducing the concept of a current graph, but, following \cite{A}, the sequence $(a_i)$ obtained from Equation (\ref{ArchEmb}) defines the face in the embedding of the current graph, see \cite{GT}, constructed (with the face-trace algorithm) starting from the edge associated with $a_1$. Here, using the notations of \cite{A}, the sequence $(\mu_i)$ defined in Equation (\ref{Memoria}) keeps track of the local sense of orientation (that is either anticlockwise or clockwise).
\end{defi}

\begin{prop}\label{CriterioCompatibility}
Given a $\W\H(m,n;h,k)$, the following facts are equivalent:
\begin{itemize}
\item[(1)] The sequence $((a_i,\mu_i))$ has period $2nk$;
\item[(2)] The first $2nk$ elements of $(a_i\mu_i)$ are all distinct;
\item[(3)] The sequence $(a_i\mu_i)$ has period $2nk$;
\item[(4)] The sequence $((a_{2i+1},\mu_{2i+1}))$ has period $nk$.
\end{itemize}
\end{prop}
\begin{proof}
Let $A$ be a $\W\H(m,n;h,k)$. Firstly, note that from Equation \eqref{ArchEmb}, it follows that:
\begin{eqnarray}\label{ArchEmb2}
a_{i-1}&=&\begin{cases}
-\lambda(\omega_r^{-\mu_{i-1}}(a_i))\omega_r^{-\mu_{i-1}}(a_i)\mbox{ if  $i$ is odd};\\
\lambda(\omega_c^{-\mu_{i-1}}(-a_i))\omega_c^{-\mu_{i-1}}(-a_i)\mbox{ if $i$ is even}.
\end{cases}
\end{eqnarray}
$(1)\Rightarrow (2)$.
Let us suppose, by contradiction, that $(1)$ holds but $(2)$ does not.
This means that there exist $i$ and $j$, with $2nk\geq j>i\geq 1$ (which implies $j-i<2nk$), such that $a_i\mu_i=a_j\mu_j$. Also, let $i$ and $j$ be indexes with this property with minimum difference.
Due to the definition of the sequences $(a_i)$ and $(\mu_i)$, if $a_i=a_j$, $\mu_i=\mu_j$ and $i\equiv j \pmod{2}$, then $T|(j-i)$ where $T=2nk$ is the period of the sequence $((a_i,\mu_i))$.
Moreover, because of the definition of $(a_i)$, if $a_i=-a_j$ we must have that $i\not\equiv j\pmod{2}$. In both cases, if $a_i\mu_i=a_j\mu_j$ and $j-i<2nk$, we must have that $i\not\equiv j\pmod{2}$. Furthermore, since in each row and each column we have more than one element, $j-i>1$.

Here we have four possible cases:
\begin{itemize}
\item[a)] $(a_i,\mu_i)=(a_j,\mu_j)$, $i$ is odd and $j$ is even;
\item[b)] $(a_i,\mu_i)=(a_j,\mu_j)$, $i$ is even and $j$ is odd;
\item[c)] $(a_i,\mu_i)=(-a_j,-\mu_j)$, $i$ is odd and $j$ is even;
\item[d)] $(a_i,\mu_i)=(-a_j,-\mu_j)$, $i$ is even and $j$ is odd.
\end{itemize}
Note that the first two cases can occur when $a_i=a_j\in \E(\Omega(A))$ while the latter two cases when $a_i=-a_j\in \pm\E(\Theta(A))$.

CASE a). Let us assume $(a_i,\mu_i)=(a_j,\mu_j)$ where $i$ is odd and $j$ is even. Then
$$a_{i+1}=-\omega_c^{\mu_{i}}(\lambda(a_i)a_i).$$
Since $\mu_{i+1}=\mu_i \lambda(-\omega_c^{\mu_{i}}(\lambda(a_i)a_i))$, it follows that
$$a_{i+1}\mu_{i+1}=-\omega_c^{\mu_{i}}(\lambda(a_i)a_i)\mu_i \lambda(-\omega_c^{\mu_{i}}(\lambda(a_i)a_i)).$$
Moreover, since $a_i\in \E(\Omega(A))$, $\lambda(a_i)=-1$, and hence we have that
$$a_{i+1}\mu_{i+1}=-\omega_c^{\mu_{i}}(-a_i)\mu_i \lambda(-\omega_c^{\mu_{i}}(-a_i)).$$

On the other hand, we have that
$$a_{j-1}=\lambda(\omega_c^{-\mu_{j-1}}(-a_j))\omega_c^{-\mu_{j-1}}(-a_j).$$
Also, since $a_j\in \E(\Omega(A))$, $\mu_j=-\mu_{j-1}$ and hence
$$a_{j-1}\mu_{j-1}=-\mu_j\lambda(\omega_c^{\mu_j}(-a_j))\omega_c^{\mu_{j}}(-a_j).$$
Finally from $(a_j,\mu_j)=(a_i,\mu_i)$ and $ \lambda(-\omega_c^{\mu_{i}}(-a_i))= \lambda(\omega_c^{\mu_{i}}(-a_i))$, it follows that:
$$a_{j-1}\mu_{j-1}=-\mu_j\lambda(\omega_c^{\mu_j}(-a_j))\omega_c^{\mu_{j}}(-a_j)=-\mu_i\lambda(-\omega_c^{\mu_i}(-a_i))\omega_c^{\mu_{i}}(-a_i)=a_{i+1}\mu_{i+1}.$$
Note that this is a contradiction since we are assuming that $i$ and $j$ are at a minimal distance and $j-i>1$.

CASE b). Let us assume $(a_i,\mu_i)=(a_j,\mu_j)$ where $i$ is even and $j$ is odd. Then,
$$a_{i+1}=\omega_r^{\mu_{i}}(-\lambda(a_i) a_i).$$
Since $\mu_{i+1}=\mu_i\lambda(\omega_r^{\mu_{i}}(-\lambda(a_i) a_i))$ and $\lambda(a_i)=-1$,  we have that
$$a_{i+1}\mu_{i+1}=\mu_i\lambda(\omega_r^{\mu_{i}}( a_i))\omega_r^{\mu_{i}}(a_i).$$

On the other hand we have that
$$a_{j-1}=-\lambda(\omega_r^{-\mu_{j-1}}(a_j))\omega_r^{-\mu_{j-1}}(a_j).$$
Also, since $a_j\in \E(\Omega(A))$, $\mu_j=-\mu_{j-1}$ and hence
$$a_{j-1}\mu_{j-1}=\mu_j\lambda(\omega_r^{\mu_{j}}(a_j))\omega_r^{\mu_{j}}(a_j).$$
Finally from  $(a_j,\mu_j)=(a_i,\mu_i)$, it follows that:
$$a_{j-1}\mu_{j-1}=\mu_j\lambda(\omega_r^{\mu_{j}}(a_j))\omega_r^{\mu_{j}}(a_j)=\mu_i\lambda(\omega_r^{\mu_{i}}(a_i))\omega_r^{\mu_{i}}(a_i)=a_{i+1}\mu_{i+1}.$$
Note that this is a contradiction since we are assuming that $i$ and $j$ are at a minimal distance and $j-i>1$.

CASE c). Let us assume $(a_i,\mu_i)=(-a_j,-\mu_j)$ where $i$ is odd and $j$ is even. Then
$$a_{i+1}=-\omega_c^{\mu_{i}}(\lambda(a_i)a_i).$$
Since $\mu_{i+1}=\mu_i \lambda(-\omega_c^{\mu_{i}}(\lambda(a_i)a_i))$ and $\lambda(a_i)=1$, we have that
$$a_{i+1}\mu_{i+1}=-\omega_c^{\mu_{i}}(a_i)\mu_i \lambda(-\omega_c^{\mu_{i}}(a_i)).$$

On the other hand, we have that
$$a_{j-1}=\lambda(\omega_c^{-\mu_{j-1}}(-a_j))\omega_c^{-\mu_{j-1}}(-a_j).$$
Also, since $a_j\in \pm \E(\Theta(A))$, $\mu_j=\mu_{j-1}$ and hence
$$a_{j-1}\mu_{j-1}=\mu_j\lambda(\omega_c^{-\mu_j}(-a_j))\omega_c^{-\mu_{j}}(-a_j).$$
Finally from  $(-a_j,-\mu_j)=(a_i,\mu_i)$ and $ \lambda(-\omega_c^{\mu_{i}}(-a_i))= \lambda(\omega_c^{\mu_{i}}(-a_i))$, it follows that:
$$a_{j-1}\mu_{j-1}=\mu_j\lambda(\omega_c^{-\mu_j}(-a_j))\omega_c^{-\mu_{j}}(-a_j)=-\mu_i\lambda(-\omega_c^{\mu_i}(a_i))\omega_c^{\mu_{i}}(a_i)=a_{i+1}\mu_{i+1}.$$
Note that this is a contradiction since we are assuming that $i$ and $j$ are at a minimal distance and $j-i>1$.

CASE d). Let us assume $(a_i,\mu_i)=(-a_j,-\mu_j)$ where $i$ is even and $j$ is odd. Then,
$$a_{i+1}=\omega_r^{\mu_{i}}(-\lambda(a_i) a_i).$$
Since $\mu_{i+1}=\mu_i\lambda(\omega_r^{\mu_{i}}(-\lambda(a_i) a_i))$ and $\lambda(a_i)=1$,  we have that
$$a_{i+1}\mu_{i+1}=\mu_i\lambda(\omega_r^{\mu_{i}}(-a_i))\omega_r^{\mu_{i}}(-a_i).$$

On the other hand we have that
$$a_{j-1}=-\lambda(\omega_r^{-\mu_{j-1}}(a_j))\omega_r^{-\mu_{j-1}}(a_j).$$
Also, since $a_j\in \pm\E(\Theta(A))$, $\mu_j=\mu_{j-1}$ and hence
$$a_{j-1}\mu_{j-1}=-\mu_j\lambda(\omega_r^{-\mu_{j}}(a_j))\omega_r^{-\mu_{j}}(a_j).$$
Finally from  $(-a_j,-\mu_j)=(a_i,\mu_i)$, it follows that:
$$a_{j-1}\mu_{j-1}=-\mu_j\lambda(\omega_r^{-\mu_{j}}(a_j))\omega_r^{-\mu_{j}}(a_j)=\mu_i\lambda(\omega_r^{\mu_{i}}(-a_i))\omega_r^{\mu_{i}}(-a_i)=a_{i+1}\mu_{i+1}.$$
Note that this is a contradiction since we are assuming that $i$ and $j$ are at a minimal distance and $j-i>1$.

Therefore $a_i\mu_i=a_j\mu_j$ can hold only assuming $j\equiv i\pmod{2}$ and hence, as noted above, $(a_i,\mu_i)=(a_j,\mu_j)$. 
It follows that $(a_{i+\ell},\mu_{i+\ell})=(a_{j+\ell},\mu_{j+\ell})$ for any positive integer $\ell$, 
which means that $T|(j-i)$ where $T=2nk$ is the period of $((a_i,\mu_i))$,
 but this is in contradiction with the assumption that $2nk\geq j>i\geq 1$. 
Thus, assuming $(1)$, i.e. $T=2nk$, the first $2nk$ elements of ($a_i\mu_i$) are distinct.

$(2)\Rightarrow (3)$. Assuming that $(2)$ holds, the first $2nk$ elements of $(a_i\mu_i)$ are different. Hence the period $T'$ of $(a_i\mu_i)$ is larger than or equal to $2nk$. Let us consider the element $a_{2nk+1}\mu_{2nk+1}$. Since it is a nonzero element of $\mathbb{Z}_{2nk+1}$, it must be equal to $a_j\mu_j$ for some $j\in [1,2nk]$. Moreover, due to the previous discussion,
$a_{2nk+1}\mu_{2nk+1}=a_j\mu_j$ can occur only if $j\equiv 2nk+1 \pmod{2}$ and $(a_i,\mu_i)=(a_j,\mu_j)$ which implies that $T|(2nk+1-j)$ where $T$ is the period of $((a_i,\mu_i))$. Note that $T$ is a multiple of the period $T'$ of the sequence $(a_i\mu_i)$. It follows that
$$2nk \leq T'\leq T\leq 2nk+1-j\leq 2nk.$$ This can occur only if $j=1$ and if the period of $(a_i\mu_i)$ is exactly $2nk$. Therefore also property $(3)$ holds.

$(3)\Rightarrow (1)$. By hypothesis, $a_1\mu_1=a_{2nk+1}\mu_{2nk+1}$. Since also $2nk+1\equiv 1\pmod{2}$, we have that $(a_1,\mu_1)=(a_{2nk+1},\mu_{2nk+1})$ which implies that $(a_{1+\ell},\mu_{1+\ell})=(a_{2nk+1+\ell},\mu_{2nk+1+\ell})$ for any positive integer $\ell$ and hence the period $T$ of $((a_i,\mu_i))$ is at most $2nk$. On the other hand, since $T$ is a multiple of the period $T'$ of $(a_i\mu_i)$, we have that
$$2nk\geq T\geq T'=2nk.$$ It follows that $T$ is exactly $2nk$ and hence $(1)$ holds.

$(1)\Rightarrow (4)$.
Since the period $T$ of $((a_i,\mu_i))$ is $2nk$, we have that $(a_{2i+1},\mu_{2i+1})$ has period $T'$ that is a divisor of $nk$. On the other hand, given $(a_{2i+1},\mu_{2i+1})=(a_{2j+1},\mu_{2j+1})$, since $2i+1\equiv 2j+1\pmod{2}$, we have that $(a_{2i+1+\ell},\mu_{2i+1+\ell})=(a_{2j+1+\ell},\mu_{2j+1+\ell})$ for any positive integer $\ell$.
It follows that $2(j-i)$ is a multiple of the period $T=2nk$ of $((a_i,\mu_i))$ and hence $(j-i)\geq nk$. Therefore the period of $(a_{2i+1},\mu_{2i+1})$ is exactly $nk$.

$(4)\Rightarrow (1)$. Since the period $T'$ of $(a_{2i+1},\mu_{2i+1})$ is $nk$, then the period $T$ of $(a_{i},\mu_{i})$ must be a multiple of $2nk$. Moreover, since $(a_{2nk+1},\mu_{2nk+1})=(a_{1},\mu_{1})$, and $1\equiv 2nk+1\pmod{2}$, we have that $(a_{2nk+1+\ell},$ $\mu_{2nk+1+\ell})=(a_{1+\ell},\mu_{1+\ell})$ for any positive integer $\ell$. Therefore, the period of $(a_{i},\mu_{i})$ must be exactly $2nk$.
\end{proof}
Now we can use the sequence $((a_i,\mu_i))$ in order to generalize the definition of compatible orderings given in \cite{A}.
\begin{defi}
Let $A$ be a $\W\H(m,n;h,k)$ and consider the ordering $\omega_r$ for the rows and $\omega_c$ for the columns. Then we say that $\omega_r$ and $\omega_c$ are \emph{compatible} whenever the conditions of Proposition \ref{CriterioCompatibility} are satisfied.
\end{defi}

\begin{defi}[Archdeacon embedding]\label{ArchdeaconEmbedding}
Let $A$ be a $\W\H(m,n;h,k)$ that admits compatible orderings $\omega_r$ and $\omega_c$.
According to Proposition \ref{CriterioCompatibility}, the following list can be seen as a cyclic permutation of $\Z_{2nk+1}\setminus \{0\}$
$$ \rho_0=(\mu_1 a_1, \mu_2 a_2,\dots, \mu_{2nk}a_{2nk}).$$

Then we define the map $\rho$ on the set of the oriented edges of the complete graph $K_{2nk+1}$ as follows
\begin{eqnarray}\label{ArchRho}
\rho((x,x+a))&=& (x,x+\rho_0(a)).
\end{eqnarray}
Clearly, given compatible orderings $\omega_r$ and $\omega_c$, the map $\rho$ is a rotation of $K_{2nk+1}$. 

Note that, considering the list $L=(a_1, a_2,\dots, a_{2nk})$, we may have repetitions.
So we define the map $\epsilon: E(K_{2nk+1})\rightarrow \{-1,1\}$ as follows:
\begin{equation}\label{epsilon}\epsilon (x,x+a)=\begin{cases}1 \mbox{ if } a \mbox{ appears once in } L;\\ -1 \mbox{ otherwise}.
\end{cases} \end{equation}
\end{defi}

Archdeacon \cite{A} proved that the following theorem holds.

\begin{thm}\label{HeffterBiemb} Let $A$ be a $\W\H(m,n;h,k)$ that admits two compatible orderings $\omega_r$ and $\omega_c$. Then there exists a cellular biembedding $\psi$ of $K_{2nk+1}$, such that every edge is on a face whose boundary length is $h$ and on a face whose boundary length is $k$. Moreover, $\psi$ is $\Z_{2nk+1}$-regular.
\end{thm}
We are also interested in describing the faces (and their lengths) induced by the Archdeacon embedding under the condition of Theorem \ref{HeffterBiemb}.

For this purpose, we take a $\W\H(m,n;h,k)$ that admits two compatible orderings $\omega_r$ and $\omega_c$.
Since an embedding of Archdeacon type can be seen as a special case of derived embedding from current graphs (see \cite{A} and \cite{GT}), given $a\in \E(R)$ for a suitable row $R$, the oriented edge $(x,x+a)$ belongs to the face $F_1$ whose boundary is
\begin{equation}\label{F1}\left(x,x+a,x+a+\omega_r(a),\ldots,x+\sum_{i=0}^{k-2} \omega_r^i(a)\right).\end{equation}
Let us now consider the oriented edge $(x,x+a)$ with $-a \in \E(C)$ for a suitable column $C$.
Then $(x,x+a)$ belongs to the face $F_2$ whose boundary is
\begin{equation}\label{F2}\left(x,x+\sum_{i=1}^{h-1}\omega_{c}^{-i}(-a),x+\sum_{i=1}^{h-2}\omega_{c}^{-i}(-a),
\dots,x+\omega_{c}^{-1}(-a)\right).\end{equation}
\begin{ex}
Let $A$ be the weak Heffter array $\W\H(3,4)$ of Example \ref{ex:WeakA}:
$$
\begin{array}{|c|c|c|c|} \hline
1 & -7 & -6 & 12 \\ \hline
2 & -4 & \pm 10 & \mp 8 \\ \hline
-3 & \mp 11 & \pm 9 & 5 \\ \hline
\end{array}
$$
On the rows of $A$, we consider the natural orderings (from left to right).
Similarly we consider the natural orderings (from top to bottom) on the columns $C_3$ and $C_4$, and its opposite on $C_1$ and $C_2$.
Here we have that the sequence $(a_1,\dots,a_{24})$ is
$$(1,3,-11,7,1,-2,-8,-5,-3,-2,-4,7,-6,10,-4,-11,9,10,-8,-12,-6,9,5,-12)$$
while $(\mu_1,\dots,\mu_{24})$ is
$$(1,1,-1,-1,-1,-1,1,1,1,1,1,1,1,-1,-1,1,-1,1,-1,-1,-1,1,1,1).$$
It follows that the rotation $\rho_0$ is given by the cycle:
$$(1,3,11,-7,-1,2,-8,-5,-3,-2,-4,7,-6,-10,4,-11,-9,10,8,12,6,9,5,-12).$$
Finally, in order to define the embedding, we need also to define the map $\epsilon: E(K_{2nk+1})\rightarrow \{-1,1\}$. According to Equation \eqref{epsilon}, this map acts as follows:
$$\epsilon (x,x+a)=\begin{cases}1 \mbox{ if } a \in\pm \{3,5\};\\ -1 \mbox{ otherwise}.
\end{cases} $$
\end{ex}
\section{An infinite class of non-orientable embeddings}
\label{sec:6r}
Embeddings of Archdeacon type into orientable surfaces have been considered by several papers (see, for instance, \cite{CDY, CMPPHeffter, CPPBiembeddings, DM}) but, apart from a single example exhibited in \cite{A}, no one has investigated such embeddings in the non-orientable case yet.
For this reason, in this section, we aim to present an infinite family of non-orientable embeddings of Archdeacon type which arise from weak Heffter arrays.
\subsection{Crazy Knight's Tour Problem}
Proposition \ref{CriterioCompatibility} leads us to consider the following problem which can be seen as a generalization of the original Crazy Knight's Tour Problem proposed in \cite{CDP}.
Given a $\W\H(m,n;h,k)$, say $A$, by $r_i$ we denote the orientation of the $i$-th row,
precisely $r_i=1$ if it is from left to right and $r_i=-1$ if it is from right to left. Analogously, for the $j$-th
column, if its orientation $c_j$ is from top to bottom then $c_j=1$ otherwise $c_j=-1$. Assume that the orientations
$\R=(r_1,\dots,r_m)$
and $\C=(c_1,\dots,c_n)$ are fixed.
Now we consider the following tour on two identical copies of the array $A$ that we denote by $A_1$ and by $A_{-1}$.
More precisely, we first denote by $Skel(A)$ the set of non-empty cells of $A$. Then we index the nonempty cells of $A_1$ with the triples $(i,j,1)$ where $(i,j)\in Skel(A)$. Similarly, we index the nonempty cells of $A_{-1}$ using the triples $(i,j,-1)$ where $(i,j)\in Skel(A)$.
Finally, given a cell $(i,j,t)$ in the array $A_t$ (here $t$ is either $1$ or $-1$) we consider the moves:
\begin{itemize}
\item[1)] $L_{\R}(i,j,t)$ is the cell $(i,j',t')$ where $j'$ is the column index of the filled cell of the row $R_{i}$ next to $(i,j)$ in the orientation $r_{i}^{t}$ and $t'=t$ if $(i,j')\in Skel(\Theta(A))$ and $t'=-t$ if $(i,j')\in Skel(\Omega(A))$.

\item[2)] $L_{\C}(i,j,t)$ is the cell $(i',j,t')$ where $i'$ is the row index of the filled cell of the column $C_j$ next to $(i,j)$ in the orientation $c_{j}^{t}$ and $t'=t$ if $(i',j)\in Skel(\Theta(A))$ and $t'=-t$ if $(i,j')\in Skel(\Omega(A))$.
\end{itemize}
Then, assuming $(1,1)\in Skel(A)$ and setting $$t=\begin{cases}1\mbox{ if }(1,1)\in Skel(\Theta(A))\\-1 \mbox{ otherwise}\end{cases}$$ we consider the list
$$ L_{\C,\R}=((1,1,t),L_{\R}\circ L_{\C}(1,1,t),\ldots,(L_{\R}\circ L_{\C})^{\ell}(1,1,t))$$
where $\ell$ is the minimum value such that $(L_{\R}\circ L_{\C})^{\ell+1}(1,1,t)=(1,1,t).$
The problem we propose here is the following.
\begin{KN}
Given a weak Heffter array $A$, do there exist $\C$ and $\R$ such that the list $L_{\C,\R}$ has length $|Skel(A)|$?
\end{KN}

Clearly the \probname\ can be stated more in general for a partially filled array $A$ and it is denoted by $P(A)$.
Here, if $L_{\C,\R}$ has length $|Skel(A)|$ we say that $(\C,\R)$ is a solution of $P(A)$.
\begin{ex}\label{solA}
Let $A$ be again the $\W\H(3,4)$ of Example \ref{ex:WeakA}.
Since the arrays $A_1$ and $A_{-1}$ are copies of $A$, we have that
$$
A_1=A_{-1}=\begin{array}{|c|c|c|c|} \hline
1 & -7 & -6 & 12 \\ \hline
2 & -4 & \pm 10 & \mp 8 \\ \hline
-3 & \mp 11 & \pm 9 & 5 \\ \hline
\end{array}
$$
Now we consider the orientations $\C=(-1,-1,1,1)$ and $\R=(1,1,1)$.
Then the list $L_{\C,\R}$ is given by:
$$((1,1,1),(3,2,-1),(1,1,-1),(2,4,1),(3,1,1),(2,2,1),(1,3,1),(2,2,-1),$$ $$(3,3,-1),(2,4,-1),(1,3,-1),(3,4,1)). $$
Since this list has length $12=|Skel(A)|$, it follows that $(\C,\R)$ is a solution of $P(A)$.

We can represent the orientations and the tour directly on the arrays $A_1$ and $A_{-1}$ as follows:
here the arrows represent the orientations ($\C$, $\R$ on $A_1$ and their opposites on $A_{-1}$), we have highlighted in grey the cells of $Skel(\Omega(A))$ and the numbers represent the positions of the cells in the tour.
$$\begin{array}{r|r|r|r|r|}
& \uparrow & \uparrow & \downarrow & \downarrow \\\hline
\rightarrow & 1 & & 7 & \\\hline
\rightarrow & & 6 & \cellcolor{gray} & \cellcolor{gray} 4 \\\hline
\rightarrow & 5 & \cellcolor{gray} & \cellcolor{gray} & 12 \\\hline
\end{array}\ \ \ \ \ \ \ \ \
\begin{array}{r|r|r|r|r|}
& \downarrow & \downarrow & \uparrow & \uparrow \\\hline
\leftarrow & 3 & & 11 & \\\hline
\leftarrow & & 8 & \cellcolor{gray} & \cellcolor{gray}10 \\\hline
\leftarrow & & \cellcolor{gray} 2 & \cellcolor{gray} 9 & \\\hline
\end{array}
$$
\end{ex}
The relationship between the Crazy Knight's Tour Problem and Archdeacon embeddings is explained by the following remark.
\begin{rem}
If $A$ is a $\W\H(m,n;h,k)$ such that $P(A)$ admits a solution $(\C,\R)$, then $A$ also admits two compatible orderings $\omega_c$ and $\omega_r$ that can be determined as follows. For each row $R$ (resp. column $C$), we consider $\omega_{R}$ to be the natural ordering if $r_i=1$ (resp. $c_i=1$) and its opposite otherwise. As usual we set $\omega_r=\omega_{R_1}\circ \cdots \circ \omega_{R_m}$ and $\omega_c=\omega_{C_1}\circ \cdots \circ \omega_{C_n}$. We consider now the sequence $(a_i)$ defined as in Equation \eqref{ArchEmb} starting from the orderings $\omega_c,\omega_r$ associated with the orientations $\C$ and $\R$. Because of the definition of $L_{\C,\R}$, the element $(i,j,t)$ in the $\ell$-th position of this list is such that the cell $(i,j)$ of $A$, considered with its row sign (i.e. its upper sign), contains $a_{2\ell+1}$ and $\mu_{2\ell+1}=t$.
It follows that, if the list $L_{\C,\R}$ has length $|Skel(A)|=nk$, then the sequence $((a_{2i+1},\mu_{2i+1}))$ has period $nk$. But, due to Proposition \ref{CriterioCompatibility}, this means that $\omega_r$ and $\omega_c$ are compatible.
\end{rem}
A consequence of this remark is that solutions of $P(A)$ provide embeddings of complete graphs, more precisely:
\begin{thm}\label{EmbeddingArco}
Let $A$ be a $\W\H(m,n;h,k)$ such that $P(A)$ admits a solution $(\C,\R)$. Then there exists a cellular biembedding $\psi$ of $K_{2nk+1}$, such that every edge is on a face whose boundary length is a multiple of $h$ and on a face whose boundary length is a multiple of $k$.
Moreover, $\psi$ is $\Z_{2nk+1}$-regular.
\end{thm}
Note that the embedding obtained from a solution $(\C,\R)$ of $P(A)$ via Theorem \ref{EmbeddingArco} is exactly the Archdeacon embedding of $K_{2nk+1}$ introduced in Definition \ref{ArchdeaconEmbedding} starting from the compatible orderings $\omega_c,\omega_r$ associated to $(\C,\R)$.

\subsection{Conditions of non-orientability}
In this subsection we present infinitely many non-orientable embeddings of Archdeacon type. For this purpose, we recall a general orientability criterion (see \cite{GT}).
\begin{thm}\label{OrCrit1}
A combinatorial embedding $\Pi=(\Gamma,\epsilon,\rho)$ is orientable if and only if any cycle $C$ of $\Gamma$ contains an even number of edges $e$ such that $\epsilon(e)=-1$ (in the following, edges of type $1$).
\end{thm}
As a consequence, we can state here the following orientability criterion for embeddings of Archdeacon type.
\begin{thm}\label{OrCrit2}
Let $\Pi=(K_{2nk+1},\epsilon,\rho)$ be an embedding of Archdeacon type. Then $\Pi$ is orientable if and only if $\epsilon(e)=1$ for all $e\in E(K_{2nk+1})$ (in the following, edges of type $0$).
\end{thm}
\proof
Let us suppose that every edge of $K_{2nk+1}$ is of type $0$. Then, any cycle of $K_{2nk+1}$ contains zero edges of type $1$. Hence, by Theorem \ref{OrCrit1}, the embedding $\Pi$ is orientable.

Now, suppose that there is an edge, say $(\bar{x},\bar{x}+a)$, of type $1$. Then, since the type of $(x,x+a)$ is independent from the value of $x$, we have that also the edges $(\bar{x}+a,\bar{x}+2a), (\bar{x}+2a,\bar{x}+3a),\dots$ are of type $1$.
Denoted by $\ell$ the additive order of $a$ in $\Z_{2nk+1}$, since $|\Z_{2nk+1}|$ is odd, we have that also $\ell$ is odd.
It follows that the cycle
$$C:=(\bar{x}, \bar{x}+a,\bar{x}+2a,\dots, \bar{x}+(\ell-1)a)$$
contains an odd number (i.e. $\ell$) of edges of type $1$. Hence, because of Theorem \ref{OrCrit1}, $\Pi$ is non-orientable.
\endproof
Now we can state a non-orientability condition that refers to the problem $P(A)$.
\begin{prop}\label{OrCrit3}
Let $A$ be a $\W\H(m,n;h,k)$, let $(\C,\R)$ be a solution of $P(A)$ and let $\Pi=(K_{2nk+1},\epsilon,\rho)$ be the associated Archdeacon embedding. If there exists $(i,j)$ such that both $(i,j,1)$ and $(i,j,-1)$ belong to $L_{\C,\R}$, then the embedding $\Pi$ is non-orientable.
\end{prop}
\proof
Let us suppose that both $(i,j,1)$ and $(i,j,-1)$ appear in the list $L_{\C,\R}$. We may assume that $(i,j,1)$ is in the $\ell$-th position of this list and $(i,j,-1)$ is in the $\ell'$-th one. Now we consider the sequence $(a_i)$ defined as in Equation \eqref{ArchEmb} with the orderings associated with the orientations $\C$ and $\R$.
Because of the definition of $L_{\C,\R}$ and since $(i,j,1)$ is the $\ell$-th element of this list, the element in the cell $(i,j)$ of $A$ considered with its row sign (i.e. its upper sign) is $a_{2\ell+1}$. On the other hand, since $(i,j,-1)$ is the $\ell'$-th element in the list $L_{\C,\R}$, the element in the cell $(i,j)$ of $A$ considered with its row sign is also $a_{2\ell'+1}$.
Therefore, we have at least one repetition in the sequence $(a_{2i+1})$ and hence, because of Definition \ref{ArchdeaconEmbedding}, there exists at least one edge $e$ such that $\epsilon(e)=-1$. It follows from Theorem \ref{OrCrit2} that the embedding $\Pi$ is not orientable.
\endproof
\begin{ex}\label{solB}
Let us consider again the weak Heffter array $A$ of Example \ref{ex:WeakA}.
We have seen in Example \ref{solA} that the orientations $\C=(-1,-1,1,1)$ and $\R=(1,1,1)$ are a solution of $P(A)$.
We recall that the list $L_{\C,\R}$ is given by:
$$((1,1,1),(3,2,-1),(1,1,-1),(2,4,1),(3,1,1),(2,2,1),(1,3,1),(2,2,-1),$$ $$(3,3,-1),(2,4,-1),(1,3,-1),(3,4,1)). $$
Since both $(1,1,-1)$ and $(1,1,1)$ appear in $L_{\C,\R}$, it follows from Proposition \ref{OrCrit3} that this solution of $P(A)$ induces a non-orientable embedding of $K_{25}$ of Archdeacon type.
\end{ex}

Now we recall an existence result of classical Heffter arrays that will be used later to obtain weak Heffter arrays and an infinite class of non-orientable embeddings of Archdeacon type.
\begin{rem}\label{ClasseWH}
As reported in Theorem \ref{th:existence}, an $\H(m,n)$ exists if and only if $m,n\geq 3$. In case $m=3$ and $n\equiv 1 \pmod{8}$, the $\H(3,n)$ presented in Theorem 3.2 of \cite{ABD} contains the elements $-1,-4k-5,4k+6$ (where $n=8k+9$) in the first row. Up to reordering the columns, we may assume these elements are in the cells $(1,1),(1,2)$ and $(1,n)$. Now, we can replace these elements with $\pm 1, \pm (4k+5), \mp (4k+6)$. Finally, transposing this array, we obtain a $\W\H(n,3)$, say $A$, such that $Skel(\Omega(A))=\{(1,1),(2,1),(n,1)\}$.
\end{rem}
\begin{prop}\label{solution}
Let $A$ be a $\W\H(n,3)$ with $n\geq 4$ such that $Skel(\Omega(A))=\{(1,1),(2,1),(n,1)\}$.
Then there exists a solution $(\C,\R)$ of $P(A)$ such that both $(1,3,1)$ and $(1,3,-1)$ belong to $L_{\C,\R}$.
\end{prop}
\begin{proof}Let us consider the orientations $\C=(1,-1,-1)$ and $\R=(-1,-1,\dots,-1)$. Since $(1,1)\in Skel(\Omega(A))$, the first element of the list $L_{\C,\R}$ is $(1,1,-1)$ and $L_{\R}\circ L_{\C}(1,1,-1)=(n,3,1)$. Then
$$L_{\R}\circ L_{\C}(n,3,1)=(n-1,2,1),$$
$$L_{\R}\circ L_{\C}(n-1,2,1)=(n-2,1,1),$$
$$L_{\R}\circ L_{\C}(n-2,1,1)=(n-1,3,1),$$
and, more in general, for any $\ell \in [1,n-5]$
$$L_{\R}\circ L_{\C}(n-\ell,3,1)=(n-1-\ell,2,1),$$
$$L_{\R}\circ L_{\C}(n-1-\ell,2,1)=(n-2-\ell,1,1),$$
$$L_{\R}\circ L_{\C}(n-2-\ell,1,1)=(n-1-\ell,3,1).$$
It follows that the list $L_{\C,\R}$ has $(4,3,1)$ in position $3n-10$ and, hence, we have that
$$L_{\C,\R}=((1,1,-1),(n,3,1),(n-1,2,1),(n-2,1,1),(n-1,3,1),\dots,(4,3,1),$$ $$(3,2,1),(2,1,-1),(1,3,1),(n,2,1),(n-1,1,1),(n,2,-1),(1,3,-1),(2,1,1),$$ $$(3,3,1),(2,2,1)).$$
Here $(2,2,1)$ is the last element since $L_{\R}\circ L_{\C}(2,2,1)=(1,1,-1)$.
It is also easy to check that the list $L_{\C,\R}$ contains $3n$ distinct elements, and hence $(\C,\R)$ is a solution of $P(A)$. Moreover both $(1,3,1)$ and $(1,3,-1)$ belong to $L_{\C,\R}$.
\end{proof}
Now we apply this solution to the problem $P(A)$ to obtain an infinite family of non-orientable embeddings.
\begin{thm}
Let $n\equiv 1 \pmod{8}$ be larger than $1$. Then there exists a non-orientable embedding of $K_{6n+1}$ of Archdeacon type.
\end{thm}
\proof
Due to Remark \ref{ClasseWH}, for any $n\equiv 1 \pmod{8}$ with $n\geq 9$, there exists a $\W\H(n,3)$, denoted by $A$, such that $Skel(\Omega(A))=\{(1,1),(2,1),(n,1)\}$ and, because of Proposition \ref{solution}, there exists a solution $(\C,\R)$ of $P(A)$ such that both $(1,3,1)$ and $(1,3,-1)$ belong to $L_{\C,\R}$.

It follows from Proposition \ref{OrCrit3} that this solution of $P(A)$ induces a non-orientable embedding of $K_{6n+1}$ of Archdeacon type.
\endproof
\section*{Conclusions}\label{sec:6}
We have several values of the parameters $n,k,t$ for which there exists an $\H_t(n;k)$, but not a strictly weak $\W\H_t(n;k)$, see Remark \ref{rem:n34}, and Propositions \ref{prop:t33} and \ref{prop:t43}.
Clearly, there is no reason to believe that these are the unique choices of the parameters for which this happens.
On the other hand, we have no example about the existence of a strictly weak $\W\H_t(n;k)$
when an $\H_t(n;k)$ does not exist. By the way, we emphasize that in Theorem \ref{prop:existence} we provide the
existence of a strictly weak $\W\H_5(n;5)$ for every $n\equiv 0\pmod 4$ with $n\geq 12$ since this case was left open in \cite{CMPP}.
Hence, up to now, the existence of an $\H_5(n;5)$ for $n\equiv 0\pmod 4$ is unknown, except for $n=8,16$, \cite{CMPP}.
For these arguments, we believe that the following question naturally arises: if there exists a strictly weak $\W\H_t(n;k)$
does then an $\H_t(n;k)$ exist too? At the moment we do not see any reasons for which the existence of a strictly weak Heffter array
implies that of a classical one. But, as already remarked above, we have no example which allow us to give a negative answer to previous question.
Hence we propose the following.

\textbf{Open Problem:} find some values of $t,n$ and $k$ for which a strictly weak $\W\H_t(n;k)$ exists, while no $\H_t(n;k)$ exists.

\section*{Acknowledgements}
The authors were partially supported by INdAM-GNSAGA.

\end{document}